\documentclass[letterpaper,12pt,reqno]{amsart}
\usepackage{graphicx,amsmath,amssymb}

\usepackage{srcltx}
\usepackage[latin1]{inputenc}
\usepackage[T1]{fontenc}

\usepackage{fullpage}

\newtheorem{X}{X}[section]

\newtheorem{lemma}[X]{Lemma}

\newtheorem{theorem}[X]{Theorem}
\newtheorem{conjecture}[X]{Conjecture}

\theoremstyle{definition}
\newtheorem{remark}[X]{Remark}
\newcommand{\V}{\text{Var}}
\newcommand{\E}{\mathbb E}

\renewcommand{\P}{\text{Prob}}
\newcommand{\R}{\mathbb R}
\newcommand{\Z}{\mathbb Z}
\newcommand{\Q}{\mathbb Q}
\newcommand{\F}{\mathbb F}

\title{Elliptic curves of unbounded rank and Chebyshev's bias}
\author{Daniel Fiorilli
}
\address{Department of Mathematics, University of Michigan, 530 Church Street, Ann Arbor MI 48109 USA}

\email{fiorilli@umich.edu}

\begin{document}

\begin{abstract}
We establish a conditional equivalence between quantitative unboundedness of the analytic rank of elliptic curves over $\mathbb Q$ and the existence of highly biased elliptic curve prime number races. 
We show that conditionally on a Riemann Hypothesis and on a hypothesis on the multiplicity of the zeros of $L(E,s)$, large analytic ranks translate into an extreme Chebyshev bias. Conversely, we show under a certain linear independence hypothesis on zeros of $L(E,s)$ that if highly biased elliptic curve prime number races do exist, then the Riemann Hypothesis holds for infinitely many elliptic curve $L$-functions and there exist elliptic curves of arbitrarily large rank.
\end{abstract}
\maketitle

\section{Introduction}

Let $E$ be a smooth elliptic curve whose minimal Weierstrass form is
\begin{equation}
E: y^2+a_1xy+a_3y = x^3+a_2x^2+a_4x+a_6
\label{equation Weierstrass form}
\end{equation} 
with $a_i \in \mathbb Z$, and let $N_E$ denote its conductor. The set of rational points on this curve $E(\mathbb Q)$ is a finitely generated abelian group by Mordell's Theorem, and hence is isomorphic to 
$$E(\mathbb Q) \cong \Z^{r} \oplus E(\Q)_{\text{tors}}, $$
where $E(\Q)_{\text{tors}}$ is the finite set of torsion points. Mazur's Theorem \cite{Maz1} gives the list of $15$ different possibilities for $E(\Q)_{\text{tors}}$. As for the integer $r=r_{\text{al}}(E)$, the algebraic rank of $E$, it is a very mysterious invariant of $E$. A central question in number theory is whether $r_{\text{al}}(E)$ is unbounded as $E$ varies.
 The highest rank found so far is due to Elkies \cite{El}, who explicitly exhibited integer coefficients $a_1, a_2, a_3, a_4$ and $a_6$ such that \eqref{equation Weierstrass form}
has algebraic rank at least $28$. It is conjectured that the set of all ranks of elliptic curves over $\mathbb Q$ is unbounded \cite{Bru,BS,FGH,Ul}. 
One approach to this conjecture is to study the $L$-function of $E$. The Birch and Swinnerton-Dyer Conjecture states that $r_{\text{al}}(E)$ is equal to the order of vanishing of $L(E,s)$ at $s=1$. Recall that the trace of the Frobenius endomorphism is given for $p\nmid N_E$ by $a_p(E)= p+1-\#E(\F_p)$, where $\#E(\F_p)$ is the number of projective points on the reduction of $E$ modulo $p$. Extending the definition of $a_p(E)$ to the whole set of primes by setting
$$ a_p(E) := \begin{cases}
1 \text{ if } E \text{ has split multiplicative reduction at }p \\
-1 \text{ if } E \text{ has nonsplit multiplicative reduction at }p \\
0 \text{ if } E \text{ has additive reduction at }p, \\
\end{cases} $$  
the $L$-function of $E$ is defined as
$$ L(E,s):= \prod_{p\mid N_E} \left( 1-\frac {a_p(E)}{p^s}\right)^{-1} \prod_{p\nmid N_E} \left( 1-\frac{a_p(E)}{p^s}+\frac p{p^{2s}}\right)^{-1}. $$

The Birch and Swinnerton-Dyer Conjecture can be seen as a local-to-global principle, since it asserts that understanding local points on $E$ is sufficient to understand its algebraic rank, which is a global invariant of $E$.
Note that the zeros of $L(E,s)$ come in conjugate pairs and are symmetric about the line $\Re(s)=1$, because of the functional equation relating $L(E,s)$ to $L(E,2-s)$.

Considering this, it is of crucial interest to understand the analytic rank $r_{\text{an}}(E)$, which by definition is the order of vanishing of $L(E,s)$ at $s=1$. 
Our main goal is to establish an equivalence between the conjecture that $r_{\text{an}}(E)$ is unbounded as $E$ varies over all elliptic curves over $\mathbb Q$ and a statement about the bias of certain prime number races formed with the local points on $E$. This is related to the initial calculations of Birch and Swinnerton-Dyer \cite{BSD1, BSD2}, who combined counts of local points on $E$ to predict its algebraic rank.

Note that if $r_{\text{an}}(E)$ is unbounded, then it should have a certain growth in terms of conductor, that is there should exist a function $f(N_E)$ tending to infinity as $N_E\rightarrow \infty$ such that

\begin{equation}
\limsup_{N_E\rightarrow \infty} \frac{r_{\text{an}}(E)}{ f(N_E)} >0.
\label{equation growth of rank}
\end{equation}  
There are two existing conjectures of this kind in the literature. Ulmer \cite{Ul} has shown the existence of non-isotrivial elliptic curves of arbitrarily large rank over the rational function field $\mathbb F_p(t)$ for which the Birch and Swinnerton-Dyer Conjecture holds. His (more general) result shows that Mestre's bound \cite{Me} on ranks of elliptic curves is best possible in the function field case. Based this result, he conjectured\footnote{Ulmer \cite{Ul} makes this conjecture on the algebraic rank, however the Birch and Swinnerton-Dyer Conjecture implies that the corresponding statement on the analytic rank is equivalent.} \cite{Ul,Ul2} that $f(N_E)=\log N_E/\log\log N_E$ is admissible in \eqref{equation growth of rank}. Note that this is also believed to be best possible, since Mestre's contidional bound on ranks of elliptic curves \cite{Me} states that
$ r_{\text{an}}(E) \ll \log N_E/\log\log N_E$. Recently, Farmer, Gonek and Hughes \cite{FGH} have developed a random matrix theory model for predicting the maximal size of the Riemann zeta function on the critical line, which led them to conjecture that
$$ \max_{0\leq t \leq T } |\zeta( \tfrac 12+it)| = e^{(1+o(1))\sqrt{\tfrac 12\log T \log\log T}}. $$	
 Their model also suggests that $f(N_E) =\sqrt{\log N_E \log\log N_E}$ is admissible and best possible in \eqref{equation growth of rank}. 

While the conjectures of Ulmer and Farmer, Gonek and Hughes are incompatible, they both imply the following weaker conjecture.
\begin{conjecture}
We have that
$$ \limsup_{N_E\rightarrow \infty} \frac{r_{\text{\emph{an}}}(E)}{\sqrt{\log N_E}} = \infty. $$
\label{conjecture unbounded ranks}
\end{conjecture}

We now describe the framework in which we relate Conjecture \ref{conjecture unbounded ranks} to elliptic curve prime number races. Chebyshev's bias is his observation in a letter to Fuss that there seems to be more primes of the form $4n+3$ than of the form $4n+1$. Chebyshev's prime number race is the study of the oscillatory quantity 
$$C(x):= \pi(x;4,3)-\pi(x;4,1), $$
which is known to have infinitely many sign changes \cite{Li}. For an account of the rich history of this subject, the reader is referred to the expository article \cite{GM}. Rubinstein and Sarnak \cite{RubSar} established under GRH and a linear independence hypothesis on the zeros of Dirichlet $L$-functions that $C(e^y)$ is positive for approximately $99.59\%$ of the values of $y$. More precisely, they established that
$$ \lim_{ Y \rightarrow \infty } \frac{\text{meas}\{ y\leq Y : C(e^y)>0 \}}Y  = 0.9959... $$

One can study a related quantity by considering local points on an elliptic curve $E$. The celebrated Hasse bound states that counting projective points,
$$ |\#E(\F_p)- (p+1)| < 2\sqrt p. $$
It is also known that the proportion of primes for which $a_p(E)= p+1-\#E(\F_p)$ is positive is equal to the proportion of primes for which it is negative. In the non-CM case, this follows from the Sato-Tate Conjecture for elliptic curves over $\mathbb Q$, recently established by Taylor, Clozel, Harris and Shepherd-Barron \cite{Tay,CHT,HST}, which states that when $E$ has no complex multiplication, the numbers $a_p(E)/2\sqrt p$ are equidistributed in $[-1,1]$ with respect to the measure $(2/\pi)\sqrt{1-t^2} dt$ (the distribution is simpler when $E$ has complex multiplication). 
Considering this, Mazur \cite{Maz} introduced the race between primes for which $a_p(E)>0$ and primes for which $a_p(E)<0$. Defining
$$ T(t):= \#\{ p\leq t : a_p(E) >0 \} - \#\{p\leq t : a_p(E) < 0\},$$
Mazur plotted the graph of $T(t)$ for various elliptic curves. The reader is encouraged to consult \cite{Maz} in which several other related quantities are studied.
Looking at the graphs appearing in Section 2.3 of \cite{Maz}, one readily sees that $T(t)$ exhibits a very erratic behaviour. Moreover it is very apparent in these plots that as $r_{\text{al}}(E)$ increases, $-T(t)$ becomes more and more biased towards positive values. Indeed, in Figure 2.5 of the paper, $-T(t)$  does not exhibit any negative value. 

Under standard hypotheses, many features of $T(t)$ were explained by Sarnak \cite{Sar}, using the explicit formula for $L(\text{Sym}^n E,s)$. Sarnak also introduced a closely related quantity $S(t)$ defined below, which as he showed can be understood using the explicit formula for $L(E,s)$ alone. Under a Riemann Hypothesis and a Linear Independence Hypothesis, Sarnak deduced an exact formula for the characteristic function of the limiting distribution of $ue^{-\frac u2}S(e^u)$, and uncovered the direct influence of the analytic rank of $E$ on this quantity.

Building on the work of Sarnak, we will study the quantity $S(t)$, which compares the primes $p$ for which $a_p(E)<0$ against those for which $a_p(E)>0$, weighted by the value of $a_p(E)/\sqrt p$. This quantity has a very similar behaviour to that of $-T(t)$. Moreover, it can be analyzed using $L(E,s)$ alone \cite{Sar}, in contrast to the analysis of $-T(t)$ which requires hypotheses on symmetric power $L$-functions of $E$.

We define the elliptic curve prime number race
$$S(t):= -\sum_{p\leq t} \frac{a_p(E)}{\sqrt p}, $$
and wish to understand the set of $t$ for which $S(t)\geq 0$. To measure the size of this set, that is to measure the bias of $S(t)$ towards either positive or negative values, we define the following lower and upper logarithmic densities:
$$ \underline{\delta}(E) := \liminf_{T\rightarrow \infty} \frac 1{\log T} \int_{ \substack{ 2\leq t\leq T:\\ S(t) \geq 0}} \frac {dt}t, \hspace{1cm} \overline{\delta}(E) := \limsup_{T\rightarrow \infty} \frac 1{\log T} \int_{ \substack{ 2\leq t\leq T:\\ S(t) \geq 0}} \frac {dt}t.$$
If these two densities are equal, then we denote them both by $\delta(E)$. 
Under ECRH and LI($E$) (see the definitions below), Sarnak \cite{Sar} has shown that $\delta(E)$ exists, and differs from $\frac 12$. In other words, $S(t)$ is always biased. Sarnak also discovered the dependence of this bias on the analytic rank of $E$, and a consequence of his results is that under ECRH and LI(E), elliptic curves of analytic rank zero have $\delta(E)<\frac 12$ (so $S(t)$ is biased towards negative values), and elliptic curves of analytic rank $\geq 1$ have $\delta(E)>\frac 12$ (so $S(t)$ is biased towards positive values).

Our first main result is that Conjecture \ref{conjecture unbounded ranks} implies that the quantity $S(t)$ can be arbitrarily biased, under the two following assumptions. Note that the second of these assumption is significantly weaker than the linear independence assumption used in \cite{Sar}.

\noindent \textbf{Hypothesis ECRH} (Elliptic Curve Riemann Hypothesis): \textit{For any elliptic curve $E$ over $\mathbb Q$, the non-trivial zeros of $L(E,s)$ have real part equal to $1$.}

\noindent \textbf{Hypothesis BM} (Bounded Multiplicity): \textit{There exists an absolute constant $C\geq 1$ such that for any elliptic curve $E$ over $\mathbb Q$, the non-real zeros of $L(E,s)$ have multiplicity at most $C$.}

\begin{theorem}[Unbounded rank $\Rightarrow$ arbitrarily biased elliptic curve prime number races]
Assume ECRH and BM, and assume Conjecture \ref{conjecture unbounded ranks} on the analytic rank of elliptic curves over $\mathbb Q$. Then for any $\epsilon>0$, there exists an elliptic curve $E_{\epsilon}$ over $\mathbb Q$ such that 
$$1-\epsilon < \underline{\delta}(E_{\epsilon})\leq \overline{\delta}(E_{\epsilon}) <1. $$
That is to say, there exists arbitrarily biased elliptic curve prime number races.
\label{theorem main}
\end{theorem}
\begin{remark}
The proof of Theorem \ref{theorem main} does not use the full strength of ECRH and BM. Indeed it is sufficient to assume that these hypotheses hold for an infinite sequence of elliptic curves $\{ E(n) \}_{n\geq 1}$ such that
$$ \lim_{n\rightarrow \infty} \frac{r_{\text{an}}(E(n))}{\sqrt{\log N_{E(n)}}}=\infty. $$
\end{remark}

Our second main result is a converse result, under a linear independence hypothesis on the zeros of $L(E,s)$. We will show that the existence of highly biased elliptic curve prime number races is very strong; under the following assumption, it implies the Riemann Hypothesis for an infinite family of $L$-functions as well as the existence of elliptic curves over $\mathbb Q$ of arbitrarily large analytic rank.

\noindent \textbf{Hypothesis LI(E)} (Linear Independence): \textit{The function $L(E,s)$ has at least one zero on the line $\Re(s)=\beta_E:=\sup \{ \Re(\rho) : L(E,\rho)=0 \}$. Moreover, the multiset $Z(E):=\{ \Im(\rho) \geq 0 : L(E,\rho)=0, \Re(\rho)=\beta_E, \rho \neq 1 \}$ is linearly independent over $\Q$.}

\begin{remark}
If the Riemann Hypothesis holds for $L(E,s)$, that is $\beta_E=\sup \{ \Re(\rho) : L(E,\rho)=0 \}=1$, then Hypothesis LI(E) implies that the multiset of all positive imaginary parts of zeros of $L(E,s)$ is linearly independent over $\mathbb Q$. The reason why only positive imaginary parts are considered is that the zeros of $L(E,s)$ come in conjugate pairs, since this $L$-function is self-dual (because $a_p(E)\in \mathbb R$). Also we can potentially have $L(E,1)=0$ (in the case $\beta_E=1$), and hence it is important not to include this zero in the multiset $Z(E)$.

If $L(E,s)$ has a non-trivial zero outside the critical line $\Re (s)=1$, that is $\beta_E>1$, then one should be careful with the additional symmetry of the set of zeros created by the functional equation. However if $\rho$ is a zero of $L(E,s)$, then the set $Z(E)=\{ \Im(\rho) \geq 0 : L(E,\rho)=0, \Re(\rho)=\beta_E, \rho \neq 1 \}$ contains at most one of the numbers $\{ \rho, \overline{\rho}, 2-\rho, 2-\overline{\rho} \}$.

Note that Hypothesis LI(E) is a hypothesis on the zeros of $L(E,s)$ lying on the line $\Re(s)=\beta_E$. In particular, if the Riemann Hypothesis does not hold for this $L$-function, then nothing is assumed on the zeros lying on the critical line. 

Finally, note that if $\beta_E>1$, then LI(E) implies that $L(E,\beta_E)\neq 0$, since a set containing zero is linearly dependent over $\mathbb Q$. This is similar to Chowla's Conjecture which states that Dirichlet $L$-functions do not vanish for $s\in (0,1]$.  

\end{remark}

\begin{theorem}[Arbitrarily biased elliptic curve prime number races $\Rightarrow$ unbounded rank]
\label{theorem converse}
Assume that there exists a sequence of elliptic curves $E_n$ over $\mathbb Q$ whose conductor tends to infinity with $n$, for which $LI(E_n)$ holds for $n\geq 1$ and whose associated prime number race is arbitrarily biased, that is as $n\rightarrow \infty$, $$\overline{\delta}(E_n) \rightarrow 1.$$
Then, there exist elliptic curves over $\mathbb Q$ of arbitrarily large analytic rank. More precisely, Conjecture \ref{conjecture unbounded ranks} holds.
\end{theorem}

\begin{remark} In the proof of Theorem \ref{theorem converse}, we actually show that under LI(E), the bias in $S(t)$ implies the Riemann Hypothesis for $L(E,s)$. The next Theorem is a precise statement of this implication.
\end{remark}

\begin{theorem}[Biased elliptic curve prime number race $\Rightarrow$ Riemann Hypothesis]
\label{theorem converse RH}

Assume that LI(E) holds and that either $\underline{\delta}(E)\neq \frac 12$ or $\overline{\delta}(E)\neq \frac 12$. Then the Riemann Hypothesis holds for $L(E,s)$.
\end{theorem}

Theorems \ref{theorem converse} and \ref{theorem converse RH} provide a method to simultaneously probe the Riemann Hypothesis and the unboundedness of the rank of elliptic curves over $\mathbb Q$. By computing the local points on an elliptic curve $E$, one can plot the prime number race $S(t)$ and if this graph is very biased towards positive values, then this gives evidence towards these two outstanding conjectures. A very strong bias is already present in the quantity $-T(t)$ associated to the rank three curve $E: y^2 + y = x^3 - 7x + 6$ appearing in Figure 2.5 of \cite{Maz}.

\begin{remark}
One can weaken the second hypothesis of Theorem \ref{theorem converse} to
$$ \limsup_{N_E \rightarrow \infty} \sqrt{\log N_E}(\overline{\delta}(E) -\tfrac 12)=\infty, $$
and still deduce the unboundedness of the analytic rank of elliptic curves over $\mathbb Q$. 

\end{remark}

\section{Proof of the necessary condition (Theorem \ref{theorem main})}
\label{section necessary condition}

We start with an outline of the proof of Theorem \ref{theorem main}. Our strategy is to show that under ECRH, the quantity
$$ E(e^y) := -\frac{y}{e^{y/2}}\sum_{p\leq e^y} \frac{a_p(E) }{\sqrt p}  $$
has a limiting distribution which is the same as the distribution of a certain random variable 
$X_E$. While we do not use this explicitly, one can see that
$$X_E = 2r_{\text{an}}(E)-1 +\sum_{\gamma > 0} \frac{2\Re (Z_{\gamma})}{\frac 14 + \gamma^2},  $$
where $\gamma$ runs over the imaginary parts of the non-trivial zeros of $L(E,s)$, and the $Z_{\gamma}$ are identically distributed random variables, uniform on the unit circle in $\mathbb C$. We will then compute the first two moments of $X_E$. While the $Z_{\gamma}$ are not necessarily independent (unless we assume a linear independence hypothesis), one can show that they have no covariance: if $\lambda > \gamma>0$, then Cov$(Z_{\gamma},Z_{\lambda})=0$. This explains the simple formula for the variance appearing in Lemma \ref{lemma first two moments}. Finally, we will see that if $r_{\text{an}}(E)$ is significantly larger than $\sqrt{\log N_E}$, then the mean of $X_E$ is significantly larger than its standard deviation, resulting in a very large bias by Chebyshev's inequality. 

The fundamental tool we will use is the explicit formula for $L(E,s)$. We start with a technical estimate for the tail of a sum over zeros of $L(E,s+\frac 12)$.

\begin{lemma}
We have for $x,T \geq 2$ that
$$ \sum_{ |\Im(\rho)| > T} \frac{x^{\rho}}{\rho} \ll \log x + \frac xT \left( (\log x)^2+ \frac{\log (TN_E)}{\log x} \right), $$
where $\rho$ runs over the non-trivial zeros of $L(E,s+\frac 12)$.
\label{lemma tail sum over zeros}
\end{lemma}
\begin{proof}
We first write
\begin{align*}
L\left(E,s+\frac 12 \right)&=\prod_{p\mid N_E} \left( 1-\frac {a_p(E)}{p^{s+\frac 12}}\right)^{-1}  \prod_{p\nmid N_E} \left(1-\frac{a_p(E)}{p^{s+\frac 12}} + \frac 1{p^{2s}} \right)^{-1} \\ &= \prod_{p\mid N_E} \left( 1-\frac {a_p(E)}{p^{s+\frac 12}}\right)^{-1}  \prod_{p\nmid N_E} \left(1-\frac{\alpha_p}{p^{s}}\right)^{-1} \left(1 - \frac {\beta_p}{p^{s}} \right)^{-1},
\end{align*}
where $\beta_p=\overline{\alpha_p}$, $|\alpha_p|=|\beta_p|=1$ and $\alpha_p+\beta_p=a_p(E)/\sqrt p$. The profound work of Wiles \cite{Wil}, Taylor and Wiles \cite{TaWi} and Breuil, Conrad, Diamond and Taylor \cite{BCDT} shows that the function $L(E,s)$ is a modular $L$-function, and hence it can be extended to an entire function. Following the proof of Theorem 6.9 of \cite{MV} we obtain using the truncated Perron Formula that for $x,T\geq 2$,
$$ \sum_{\substack{p^e\leq x \\ e\geq 1 \\ p\nmid N_E}} (\alpha_p^e+\beta_p^e) \log p +\sum_{\substack{ p^e \leq x \\ e\geq 1 \\ p\mid N_E}} \frac{a_p(E)^e}{p^{e/2}} \log p =\int_{\substack{\Re (s)=1+\frac 1{\log x} \\ |\Im (s)|\leq T}} \frac{L'(E,s+\frac 12)}{L(E,s+\frac 12)}x^s \frac{ds}{s}+ O\left( \log x + \frac xT (\log x)^2 \right).$$
In particular we have that
$$  \sum_{\substack{p^e\leq x \\ e\geq 1 \\ p\nmid N_E}} (\alpha_p^e+\beta_p^e) \log p +\sum_{\substack{ p^e \leq x \\ e\geq 1 \\ p\mid N_E}} \frac{a_p(E)^e}{p^{e/2}}  \log p  =\int_{\substack{\Re (s)=1+\frac 1{\log x} }} \frac{L'(E,s+\frac 12)}{L(E,s+\frac 12)}x^s \frac{ds}{s}+ O( \log x ), $$
and hence subtracting these two estimates we obtain
\begin{equation}
\int_{\substack{\Re (s)=1+\frac 1{\log x} \\ |\Im (s)| > T}} \frac{L'(E,s+\frac 12)}{L(E,s+\frac 12)}x^s \frac{ds}{s}  \ll \log x+ \frac xT (\log x)^2.  
\label{equation top of integral}
\end{equation}
Now, if $s$ is at a distance $\gg (\log N_E)^{-1}$ from the zeros of $L(E,s+\frac 12)$, then (5.27) and (5.28) of \cite{IwKo} give the bound 
$$ \frac{L'(E,s+\frac 12)}{L(E,s+\frac 12)} \ll \log (N_E (|\Im(s)|+2)). $$
Using this bound, \eqref{equation top of integral} becomes
$$  \sum_{ |\Im(\rho)| > T} \frac{x^{\rho}}{\rho}  +  \int_{\substack{ -\infty<\Re(s)\leq 1+\frac 1{\log x} \\ |\Im (s)|=T }}  \frac{L'(E,s+\frac 12)}{L(E,s+\frac 12)}x^s \frac{ds}{s} \ll \log x+ \frac xT (\log x)^2,$$
where the contour of integration of the last integral should be slightly perturbed to a contour $C$ which is at a distance $\gg (\log (N_ET))^{-1}$ from each zero of $L(E,s)$ (this is possible thanks to the Riemann-von Mangoldt Formula and the zero-free region of $L(E,s)$), and thus
$$\sum_{ |\Im(\rho)| > T} \frac{x^{\rho}}{\rho}  \ll  \int_{C} \log (N_E T)  x^{\Re(s)} \frac{|ds|}{|s|} +\log x+ \frac xT (\log x)^2\ll \log x + \frac xT \left( \frac{\log (N_E T)}{\log x} + (\log x)^2\right). $$

\end{proof}

The main tool we will use is the explicit formula (see the corresponding (13) of \cite{Sar}).
\begin{lemma}
\label{lemma explicit formula}
Assume the Riemann Hypothesis for $L(E,s)$. Then we have for $x,T\geq 2$ that
\begin{align}
E(x):&=-\frac{\log x}{\sqrt x}\sum_{p\leq x} \frac{a_p(E) }{\sqrt p} = 2r_{\text{an}}(E)-1+\sum_{\gamma_E \neq 0} \frac{e^{i\gamma_E \log x}}{\frac 12+i\gamma_E} + o_E(1)\label{equation lemma explicit formula}
\\
&= 2r_{\text{an}}(E)-1+\sum_{0<|\gamma_E| \leq T} \frac{e^{i\gamma_E \log x}}{\frac 12+i\gamma_E} +O\left(\frac {\sqrt x}T (\log(xT N_E))^2 \right)+ o_{x\rightarrow \infty}(1),
\label{equation lemma explicit formula with error term}
\end{align}
where $\gamma_E$ runs over the imaginary parts of the non-trivial zeros of $L(E,s)$.
\end{lemma}
\begin{proof}
We start with the explicit formula for 
$$
L\left(E,s+\frac 12 \right)= \prod_{p\mid N_E} \left( 1-\frac {a_p(E)}{p^{s+\frac 12}}\right)^{-1}  \prod_{p\nmid N_E} \left(1-\frac{\alpha_p}{p^{s}}\right)^{-1} \left(1 - \frac {\beta_p}{p^{s}} \right)^{-1},
$$
where as before, $\beta_p=\overline{\alpha_p}$, $|\alpha_p|=|\beta_p|=1$ and $\alpha_p+\beta_p=a_p(E)/\sqrt p$. 
Taking $T=x$ in (5.53) of \cite{IwKo} and bounding the rest of the sum over zeros using Lemma \ref{lemma tail sum over zeros} we obtain the estimate
\begin{equation}
\sum_{\substack{p^e\leq x \\ e\geq 1 \\ p\nmid N_E}} (\alpha_p^e+\beta_p^e) \log p = -\sum_{\gamma_E} \frac{x^{\frac 12+i\gamma_{E}}}{ \frac 12+i\gamma_E} +O(\log x \log(xN_E)).
\label{equation formule explicite}
\end{equation} 
Using the trivial bound on the terms on the left-hand side with $e\geq 3$, this becomes
\begin{multline}
-x^{-\frac 12}\sum_{p\leq x} \frac{a_p(E)}{\sqrt p} \log p = x^{-\frac 12}\sum_{p\leq \sqrt x} (\alpha_p^2 + \beta_p^2) \log p + 2r_{\text{an}}(E) + \sum_{\gamma_E\neq 0} \frac{e^{i\gamma_E \log x}}{\frac 12+i\gamma_E} +O_E(x^{-\frac 16}).
\label{equation squares to determine}
\end{multline} 
Now, $L(\text{Sym}^2E,s+1)$ is holomorphic at $s=1$, and a Tauberian argument shows that
$$  \sum_{p\leq \sqrt x} (\alpha_p^2 +\alpha_p\beta_p+ \beta_p^2) \log p = o_E(\sqrt x), $$
which combined with \eqref{equation squares to determine},  the fact that $\alpha_p\beta_p=1$ and the Prime Number Theorem gives
$$ -x^{-\frac 12}\sum_{p\leq x} \frac{a_p(E)}{\sqrt p} \log p =  2r_{\text{an}}(E)-1+ \sum_{\gamma_E\neq 0} \frac{e^{i\gamma_E \log x}}{\frac 12+i\gamma_E} +o_{E}(1).   $$
The estimate \eqref{equation lemma explicit formula} follows by a summation by parts as in Lemma 2.1 of \cite{RubSar}, and \eqref{equation lemma explicit formula with error term} follows by applying Lemma \ref{lemma tail sum over zeros}.

\end{proof}

\begin{lemma}
Assume the Riemann Hypothesis for $L(E,s)$. Then the quantity $E(x)$ defined in Lemma \ref{lemma explicit formula} has a limiting logarithmic distribution, that is there exists a Borel measure $\mu_E$ on $\R$ such that for any bounded Lipschitz continuous function $f:\R \rightarrow \R$ we have
\begin{equation}
\lim_{Y\rightarrow \infty} \frac 1Y \int_2^{Y} f(E(e^y)) dy = \int_{\mathbb R} f(t) d\mu_E(t).
\label{equation lemma limiting distribution}
\end{equation} 
\label{lemma limiting distribution}
\end{lemma}
\begin{proof}
This follows from \cite{ANS}.
\end{proof}

\begin{remark}
By taking $f$ to be identically one in \eqref{equation lemma limiting distribution} we deduce that $\mu_E(\mathbb R) =1$, that is $\mu_E$ is a probability measure.
\end{remark}

Let $X_E$ be the random variable associated to $\mu_E$. We will show that the moments of $E(e^y)$ agree with those of $X_E$.

\begin{lemma}
Assume the Riemann Hypothesis for $L(E,s)$. We have for $k\geq 1$ that 
$$ \lim_{Y\rightarrow \infty} \frac 1Y \int_2^{Y} E(e^y)^k dy = \int_{\mathbb R} t^k d\mu_E(t). $$
\label{lemma moments in terms of E(y)}
\end{lemma}

\begin{proof}

We will only prove the $k=1$ case since the general result follows along the same lines. Let $S\geq 1$ and define the bounded Lipschitz continuous function
$$H_S(t):= \begin{cases}
0 &\text{ if } |t| \leq S \\
|t|-S &\text{ if } S < |t| \leq S+1 \\
1 &\text{ if } |t|>S+1.
\end{cases} $$
By Lemma \ref{lemma limiting distribution} we have that
$$ \lim_{Y\rightarrow \infty} \frac 1Y \int_2^Y H_S(E(e^y)) dy = \int_{\mathbb R} H_S(t) d\mu_E(t) \leq \mu_E((-\infty,-S] \cup [S,\infty)).  $$
In a similar way to Theorem 1.2 of \cite{RubSar}, one can show that $\mu_E$ has exponentially small tails: 
$$ \mu_E((-\infty,-S] \cup [S,\infty)) \ll_E \exp(-c_E\sqrt S),$$
from which we obtain that
$$\limsup_{Y\rightarrow \infty} \frac 1Y \int_{\substack{2\leq y \leq Y \\ |E(e^y)| > S+1}} dy \leq \lim_{Y\rightarrow \infty} \frac 1Y \int_2^Y H_S(E(e^y)) dy  \ll_E \exp(-c_E\sqrt S). $$
By using dyadic intervals we easily show that this bound implies
\begin{equation}
\limsup_{Y\rightarrow \infty} \frac 1Y \int_{\substack{2\leq y \leq Y \\ |E(e^y)| \geq S}} |E(e^y)| dy \ll_E  \exp(-c'_E\sqrt S). \label{equation bound on tails}
\end{equation}
Therefore, defining the bounded Lipschitz continuous function 
$$G_S(t):= \begin{cases}
t &\text{ if } |t| \leq S \\
S(S+1-|t|) &\text{ if } S < |t| \leq S+1 \\
0 &\text{ if } |t|>S+1,
\end{cases} $$
we obtain using \eqref{equation bound on tails} and Lemma \ref{lemma limiting distribution} that
\begin{multline*} \limsup_{Y\rightarrow \infty} \frac 1Y \int_2^Y E(e^y) dy = \limsup_{Y\rightarrow \infty} \frac 1Y \int_2^Y G_S(E(e^y)) dy + O_E(\exp(-c''_E\sqrt S) ) 
\\ = \int_{\mathbb R} G_S(t)d\mu_E(t) + O_E(\exp(-c''_E\sqrt S) ) =  \int_{\mathbb R} td\mu_E(t) + O_E(\exp(-c'''_E\sqrt S) ),
 \end{multline*}
 and so taking $S\rightarrow \infty$ we obtain that 
 $$ \limsup_{Y\rightarrow \infty} \frac 1Y \int_2^Y E(e^y) dy = \int_{\mathbb R} td\mu_E(t). $$
The same argument works for the $\liminf$, and hence the assertion is proved. 
\end{proof}

We now explicitly compute the first two moments of $X_E$, the random variable associated to the measure $\mu_E$. This is analogous to Schlage-Puchta's result \cite{Pu}.
\begin{lemma}
\label{lemma first two moments}
Assume the Riemann Hypothesis for $L(E,s)$. Then,
$$ \E[X_E]= 2r_{\text{an}}(E)-1, \hspace{1cm} \text{\emph{Var}}[X_E] = \sum_{\gamma_E\neq 0}^* \frac{m(\gamma_E)^2}{\frac 14+\gamma_E^2},$$
where the last sum runs over the imaginary parts of the non-trivial zeros of $L(E,s)$, the star meaning that we count the zeros without multiplicity, and
$m(\gamma_E)$ denotes the multiplicity of the zero $\rho_E=1+i\gamma_E$.

\end{lemma}
\begin{proof}
By Lemma \ref{lemma explicit formula} we have that 
\begin{align*}\int_2^{Y} E(e^y) dy &= (2r_{\text{an}}(E)-1)(Y-2) + \sum_{\gamma_E\neq 0} \frac 1{ \frac 12+i\gamma_E} \int_2^{Y} e^{i\gamma_E y} dy +o(Y)  \\
&= (2r_{\text{an}}(E)-1)(Y-2) + O_E\left(  1 \right)+o(Y),
\end{align*}
since the sum $\sum_{\gamma_E\neq 0} \frac 1{|\frac 12+i\gamma_E|\gamma_E}$ converges. Taking $Y \rightarrow \infty$ and applying Lemma \ref{lemma moments in terms of E(y)} gives that
$$\E[X_E] = \lim_{Y\rightarrow \infty} \frac 1Y \int_2^{Y} E(e^y) dy= 2r_{\text{an}}(E)-1. $$

As for the second assertion, it follows from Plancherel's identity for Besicovitch $B^2$ almost-periodic functions. In an effort to be more self-contained we include a proof which follows \cite{Pu}. We use Lemma \ref{lemma explicit formula} again. Letting $\gamma$ and $\lambda$ run trough the ordinates of the non-trivial zeros of $L(E,s)$, we have

\begin{align*} \int_2^{Y} \left| \sum_{0<|\gamma_E|\leq T} \frac{e^{i\gamma_E y}}{\frac 12+i\gamma_E} \right|^2 dy &= \sum^*_{\substack{ 0< |\gamma|, |\lambda| \leq T }} \frac {m(\gamma)m(\lambda)}{(\frac 12+i\gamma)(\frac 12-i\lambda)} \int_2^{Y} e^{iy(\gamma-\lambda)} dy \\ &= (Y-2) \sum^*_{0< |\gamma_E|\leq T} \frac{m(\gamma_E)^2}{\frac 14+\gamma_E^2} +O\left( \sum_{\substack{ 0<|\gamma|,|\lambda| \leq T \\ \gamma\neq \lambda}} \frac {\min\{Y, |\gamma-\lambda|^{-1}\}}{(1+|\gamma|)(1+|\lambda|)}   \right).
\end{align*}
(Note that we have removed the star in the sum in the error term, which explains why the multiplicities dissapeared.)
The first sum converges absolutely, since the Riemann-von Mangoldt formula (see Theorem 5.8 of \cite{IwKo})
$$ N(T,E) :=\#\{\gamma_E : 0\leq \gamma_E \leq T \}= \frac T{2\pi} \log \left(\frac{N_E T}{2\pi e} \right)  + O(\log (N_E T))  $$
implies that $m(\gamma_E) \ll \log(N_E(3+|\gamma_E|))$. Introducing a parameter $1\leq U< T$, the sum appearing in the error term is at most:
\begin{align}
&\sum_{\substack{0<|\gamma|,|\lambda| \leq T \\ |\gamma-\lambda|\geq 1 }} \frac {|\gamma-\lambda|^{-1}}{(1+|\gamma|)(1+|\lambda|)} + \sum_{\substack{U<|\gamma|,|\lambda| \leq T \\ |\gamma-\lambda|\leq 1 }} \frac {Y}{(1+|\gamma|)(1+|\lambda|)}+\sum_{\substack{0<|\gamma|,|\lambda| \leq U \\ 0<|\gamma-\lambda|\leq 1 }} \frac {|\gamma-\lambda|^{-1}}{(1+|\gamma|)(1+|\lambda|)} \notag \\
&\ll_E \sum_{\substack{\gamma,\lambda\\ |\gamma-\lambda|\geq 1 }} \frac {|\gamma-\lambda|^{-1}}{(1+|\gamma|)(1+|\lambda|)} + Y \sum_{U\leq |\gamma| \leq T} \frac{\log |\gamma|}{ (1+|\gamma|)^2} + S(U) \notag\\
& \ll_E 1 + Y \frac{(\log U)^2}{U} + S(U), \label{last equation in error term}
\end{align} 
since the integral $\iint_{|x-y|\geq 1} \frac{|x-y|^{-1} \log x \log y}{(|x|+1)(|y|+1)} dxdy$ converges. Here, $$S(U):=\sum_{\substack{0<|\gamma|,|\lambda| \leq U \\ 0<|\gamma-\lambda|\leq 1 }} \frac {|\gamma-\lambda|^{-1}}{(1+|\gamma|)(1+|\lambda|)}.$$ 

Define $Y_U\geq U^2$ to be an increasing function of $U$ such that for each $U\geq 1$, $US(U)\leq Y_U$ (this is ineffective). Inverting this process, we find an increasing function $U(Y)\leq \sqrt Y$ such that $U(Y)\rightarrow \infty$ as $Y\rightarrow \infty$, and such that for $Y$ large enough, $U(Y)S(U(Y)) \leq Y$. This shows that \eqref{last equation in error term} is 
$$ \ll 1 + Y \frac{(\log U(Y))^2}{U(Y)} + \frac{Y}{U(Y)} = o_{Y \rightarrow \infty}(Y).  $$
That is, we have shown that
 $$\int_2^{Y} \left| \sum_{0<|\gamma_E|\leq T} \frac{e^{i\gamma_E y}}{\frac 12+i\gamma_E} \right|^2 dy = (Y-2) \sum^*_{0< |\gamma_E|\leq T} \frac{m(\gamma_E)^2}{\frac 14+\gamma_E^2} +o_{Y \rightarrow \infty}(Y).$$
Therefore, by Lemma \ref{lemma explicit formula} we obtain by taking $T=e^{2Y}$ that
\begin{align*} \int_2^{Y} |E(e^y) - \E[X_E]|^2 dy &= \int_2^{Y} \left| \sum_{0<|\gamma_E|\leq e^{2Y}} \frac{e^{i\gamma_E y}}{\frac 12+i\gamma_E} \right|^2 dy \\
&\hspace{2cm} + O_E\left( \int_2^Y\left| \sum_{0<|\gamma_E|\leq e^{2Y}} \frac{e^{i\gamma_E y}}{\frac 12+i\gamma_E} \right| o_{y\rightarrow \infty}(1) dy  + o_{Y\rightarrow \infty}(Y)\right)
\\&= (Y-2)\sum^*_{0<|\gamma_E| \leq e^{2Y}} \frac {m(\gamma_E)^2}{\frac 14+\gamma_E^2} +  o_{Y\rightarrow \infty}(Y),
\end{align*}
by the Cauchy-Schwartz inequality. 
The result follows by taking $Y\rightarrow \infty$ and applying Lemma \ref{lemma moments in terms of E(y)}.

\end{proof}

\begin{lemma}
\label{lemma bias factor}
Assume the Riemann Hypothesis for $L(E,s)$. If $$B(E):=\frac{\E[X_E]}{\sqrt{\V[X_E]}}$$
is large enough, then 
$$  \underline{\delta}(E) \geq 1-2\frac{\V[X_E]}{\E[X_E]^2}. $$

\end{lemma}
\begin{proof}
It is clear from Lemma \ref{lemma first two moments} and the Riemann-von Mangoldt formula that 
$\V[X_E] \gg \log N_E$, and therefore our assumption that $B(E)$ is large enough implies that $\E[X_E]$ is also large enough, say at least $4$. 
Let now
$$H(x):= \begin{cases} 1  &\text{ if } x \geq 0 \\ 0 &\text{ if } x<0,
\end{cases} \hspace{1cm}
f(x):= \begin{cases} 1 &\text{ if } x\geq 1\\
x &\text{ if } 0\leq x \leq 1 \\
0 &\text{ if } x <0. \\ 
\end{cases}
 $$ 
Clearly, $f(x)$ is bounded Lipschitz continuous and $f(x)\leq H(x)$. Therefore,

$$ \underline{\delta}(E) = \liminf_{Y\rightarrow \infty} \frac 1Y \int_2^Y H(E(e^y)) dy \geq  \liminf_{Y\rightarrow \infty} \frac 1Y \int_2^Y f(E(e^y)) dy, $$
which by Lemma \ref{lemma limiting distribution} is equal to
\begin{align*} \int_{\mathbb R} f(t) d\mu_E(t)&= 1-\int_{\mathbb R} (1-f(t)) d\mu_E(t)\\
&=1-\int_{-\infty}^{1} (1-f(t)) d\mu_E(t)\geq 1-\mu_E(-\infty,1]. 
\end{align*}

We now apply Chebyshev's inequality:
\begin{multline*}
\mu_E(-\infty,1] = \P[X_E \leq 1]  = \P[X_E-\E[X_E] \leq 1-\E[X_E]]\\ \leq \P[|X_E-\E[X_E]| \geq \E[X_E]-1] 
 \leq \frac{\V[X_E]}{(\E[X_E]-1)^2} \leq 2\frac{\V[X_E]}{\E[X_E]^2} 
\end{multline*}
since $\E[X_E]\geq 4$, and therefore 
$$\underline{\delta}(E) \geq 1-2\frac{\V[X_E]}{\E[X_E]^2}. $$

\end{proof}

We are now ready to prove Theorem \ref{theorem main}.

\begin{proof}[Proof of Theorem \ref{theorem main}]
Let $X_E$ be the random variable associated to the measure $\mu_E$. By Lemma \ref{lemma first two moments}, its mean is equal to $\E[X_E]= 2r_{\text{an}}(E)-1$, and by our assumption that the non-real zeros of $L(E,s)$ have bounded multiplicity we have
$$\sum_{\gamma_E\neq 0} \frac{1}{\frac 14+\gamma_E^2} \leq \V[X_E] = \sum_{\gamma_E\neq 0}^* \frac{m(\gamma_E)^2}{\frac 14+\gamma_E^2} \ll \sum_{\gamma_E\neq 0} \frac{1}{\frac 14+\gamma_E^2}, $$
and thus by the Riemann-von Mangoldt formula,
$$ \V[X_E] \asymp \log N_E.$$
The condition 
$$  \limsup_{N_E \rightarrow \infty} \frac{r_{\text{an}}(E)}{\sqrt{\log N_E}} = \infty$$
then implies that 
$$ \limsup_{N_E \rightarrow \infty} \frac{\E[X_E]}{ \sqrt{\V[X_E]}} = \infty. $$
Combining this with Lemma \ref{lemma bias factor} shows that 
$$ \sup_{E} \underline{\delta}(E)=1. $$
The last thing to show is that $\overline{\delta}(E)<1$ for all elliptic curves, however this follows from an analysis as in Lemma \ref{lemma bias factor} combined with a lower bound on $\mu_E(-\infty,-1]$ similar to that in Theorem 1.2 of \cite{RubSar}, which can be derived using similar techniques.

\end{proof}

\section{Proof of the sufficient condition (Theorem \ref{theorem converse})}

The first step will be to show that under the first assumption in Theorem \ref{theorem converse}, the quantity
$$E_n(x) := -\frac{\log x}{x^{\beta_n-\frac 12}} \sum_{p\leq x} \frac{a_p(E_n)}{\sqrt p}$$ 
has a limiting logarithmic distribution as $x\rightarrow \infty$. 

\begin{lemma}
\label{lemma converse truncated sum}
Fix $T\geq 1$ and assume that $L(E,s)$ has at least one zero on the line $\Re(s)=\beta_0:=\sup\{ \Re(z) : L(E,z)=0 \}$. Then letting $\rho$ run over the non-trivial zeros of $L(E,s+\frac 12)$ we have that the quantity
$$ F_T(x):=x^{-\beta_0+\frac 12}\sum_{\substack{\rho \\ |\Im(\rho)|\leq T}} \frac{x^{\rho}}{\rho }  -x^{1-\beta_0}$$
has a limiting logarithmic distribution as $x\rightarrow \infty$, that is there exists a Borel measure $\mu^{(T)}_E$ such that for every bounded Lipschitz continuous function $f$ we have
$$ \lim_{Y\rightarrow \infty} \frac 1Y \int_2^Y f(F_T(e^y))dy = \int_{\mathbb R} f(t) d\mu_E^{(T)}(t). $$

\end{lemma}

\begin{remark}
Since the non-trivial zeros of $L(E,s)$ are symmetric about the line $\Re (s)=1$, we always have $\beta_0\geq 1$. The Riemann Hypothesis for $L(E,s)$ states that $\beta_0=1$. 
\end{remark}

\begin{proof}[Proof of Lemma \ref{lemma converse truncated sum}]
Define $\beta_T:= \sup\{  \Re(z) : L(E,z)=0, |\Im (z)|\leq T, \Re(z) <\beta_0  \}$, which is strictly less than $\beta_0$ since $L(E,s)$ has only finitely many zeros of height at most $T$. We have by the Riemann-von Mangoldt Formula that
\begin{equation}
F_T(x)= \sum_{\substack{\rho \\ \Re(\rho)=\beta_0-\frac 12 \\ |\Im(\rho)|\leq T}} \frac{x^{i\gamma}}{\rho }-x^{1-\beta_0} +O_E(x^{-\delta_T} (\log T)^2),  
\label{equation F_T approximation}
\end{equation} 
where $\delta_T=\beta_0-\beta_T >0$ and $\rho=\eta+i\gamma$ runs over the non-trivial zeros of $L(E,s+\frac 12)$. Hence, the limiting logarithmic distribution of $F_T(x)$ coincides with the limiting distribution of  
$$ G_T(y):=- \epsilon(\beta_0)+\sum_{\substack{\rho  \\ \Re(\rho)=\beta_0-\frac 12\\ |\Im(\rho)|\leq T}} \frac{e^{i \gamma y}}{\rho }, $$
which exists by arguments analogous to Lemma 2.3 of \cite{RubSar}. Here,
$$\epsilon(\beta_0) := \begin{cases}
  1 &\text{ if } \beta_0=1 \\
0 &\text{ otherwise.}
\end{cases}$$

\end{proof}

We now adapt Lemma 2.2 of \cite{RubSar}.

\begin{lemma}
\label{lemma converse epsilon_T}
Let 
$$ \epsilon(x;T):=  x^{-\beta_0+\frac 12} \sum_{\substack{\rho \\  |\Im(\rho)|> T}} \frac{x^{\rho}}{\rho }, $$
where $\beta_0=\sup\{ \Re(z) : L(E,z)=0 \}$ and $\rho$ runs over the non-trivial zeros of $L(E,s+\frac 12)$. Then we have for $T\geq 1$ and $Y \geq 2$ that
$$ \int_2^Y |\epsilon(e^y;T)|^2 dy \ll_E Y\frac{(\log T)^2}T + \frac{(\log T)^3}T. $$

\end{lemma}

\begin{proof}

We compute
\begin{align*}
\int_2^Y |\epsilon(e^y;T)|^2 dy  &= \sum_{\substack{\rho_1,\rho_2 \\ |\Im(\rho_1)|,|\Im(\rho_2)| >T}} \int_2^Y \frac{e^{y(\rho_1+\overline{\rho}_2-2\beta_0+1)}}{\rho_1\overline{\rho}_2} dy  \\
& \ll_E \sum_{\substack{\rho_1,\rho_2 \\ |\Im(\rho_1)|, |\Im(\rho_2)| >T}}  \frac 1{|\Im(\rho_1)||\Im(\rho_2)|} \min(Y,|\rho_1+\overline{\rho}_2-2\beta_0+1|^{-1}),
\end{align*}
since one can easily show that for any $s\in \mathbb C$ with $\Re(s)\leq0$,
$$  \left|\int_2^Y e^{sy} dy \right| \leq \min(2|s|^{-1},Y).   $$
The proof follows as in Lemma 2.2 of \cite{RubSar} since $|\rho_1+\overline{\rho}_2-2\beta_0+1|^{-1} \leq |\Im(\rho_1)-\Im(\rho_2)|^{-1}$.

\end{proof}

\begin{lemma}
\label{lemma limiting distribution of E converse}
Assume that $L(E,s)$ has at least one zero on the line $\Re(s)=\beta_0:=\sup\{ \Re(z) : L(E,z)=0 \}\geq 1$. Then the quantity
$$ E(x) := -\frac{\log x}{x^{\beta_0-\frac 12}} \sum_{p\leq x} \frac{a_p(E)}{\sqrt p}
$$
has a limiting logarithmic distribution $\mu_E$ as $x\rightarrow \infty$.
\end{lemma}

\begin{proof}

We argue as in Lemma \ref{lemma explicit formula}. Defining $\alpha_p$ and $\beta_p$ as we did in this lemma we have by (5.53) of \cite{IwKo} that for $1\leq U\leq x$ (see the corresponding (13) of \cite{Sar}),
$$ \sum_{\substack{p^e \leq x \\ e\geq 1 \\ p\nmid N_E}} (\alpha_p^e+\beta_p^e) \log p = - \sum_{|\Im(\rho)|\leq U} \frac{x^{\rho}}{\rho} +O\left(\frac xU\log x \log(N_Ex)\right),  $$
where $\rho$ runs over the non-trivial zeros of $L(E,s+\frac 12)$, and hence
\begin{multline*} -x^{-\beta_0+\frac 12}\sum_{p\leq x} \frac{a_p(E)}{\sqrt p} \log p = x^{-\beta_0+\frac 12}\sum_{p\leq \sqrt x} (\alpha_p^2 + \beta_p^2) \log p +x^{-\beta_0+\frac 12} \sum_{|\Im(\rho)|\leq U} \frac{x^{\rho}}{\rho}\\ +O_E\left( \frac{x^{\frac 32-\beta_0}}U \log x \log(N_E x) +x^{\frac 56-\beta_0}\right).  
\end{multline*}
As in Lemma \ref{lemma explicit formula}, combining this with a Tauberian argument on $L(\text{Sym}^2E,s+1)$ and a summation by parts gives that
$$ E(x) =  x^{-\beta_0+\frac 12} \sum_{|\Im(\rho)|\leq U} \frac{x^{\rho}}{\rho}-x^{1-\beta_0} +O_E \left( \frac{x^{\frac 32-\beta_0}}U (\log x)^2 \right) +o_{x\rightarrow \infty}(1),$$
and so taking $U=x$, using that $\beta_0\geq 1$, and applying Lemma \ref{lemma tail sum over zeros} we obtain that for any $T\geq 1$,
$$ E(x) = x^{-\beta_0+\frac 12} \sum_{\rho} \frac{x^{\rho}}{\rho}-x^{1-\beta_0} +o_{x\rightarrow \infty}(1)=F_T(x)+\epsilon(x;T)+o_{x\rightarrow \infty}(1), $$
where $F_T(x)$ and $\epsilon(x;T)$ are defined in Lemmas \ref{lemma converse truncated sum} and \ref{lemma converse epsilon_T} respectively. 

Let now $f$ be a bounded Lipschitz continuous function. We have as in Section 2.1 of \cite{RubSar} that
 \begin{align*}
  \frac 1Y \int_{2}^Y f(E(e^y)) dy &= \frac 1Y \int_{2}^Y f(F_T(e^y)) dy +O_f\left(\frac 1{Y} \int_2^Y|\epsilon(e^y;T)| dy \right) + o_{Y\rightarrow \infty}(1) \\
  &= \frac 1Y \int_{2}^Y f(F_T(e^y)) dy +O_f\left(\frac 1{\sqrt Y} \left(\int_2^Y|\epsilon(e^y;T)|^2 dy\right)^{\frac 12} \right) + o_{Y\rightarrow \infty}(1)\\
  &=\frac 1Y \int_{2}^Y f(F_T(e^y)) dy +O_f\left( \frac{\log T}{\sqrt T} + \frac{(\log T)^{\frac 32}}{\sqrt {TY}}\right)+ o_{Y\rightarrow \infty}(1)
\end{align*} 
by the Cauchy-Schwartz inequality and Lemma \ref{lemma converse epsilon_T}.
Taking $Y\rightarrow \infty$ and using that $F_T(e^y)$ has a limiting distribution (Lemma \ref{lemma converse truncated sum}) we obtain
\begin{multline}
\int_{\mathbb R} f(x) d\mu_E^{(T)}(x) -O\left( \frac{\log T}{\sqrt T}\right)\leq    \liminf_{Y\rightarrow \infty}  \frac 1Y \int_{2}^Y f(E(e^y)) dy \\\leq \limsup_{Y\rightarrow \infty}  \frac 1Y \int_{2}^Y f(E(e^y)) dy \leq \int_{\mathbb R} f(x) d\mu_E^{(T)}(x) +O\left( \frac{\log T}{\sqrt T}\right). 
\label{equation squeezing measures}
\end{multline}
As in \cite{ANS}, we apply Helly's Theorem to the sequence of probability measures $\{\mu_E^{(T)}\}_{T\geq 1}$; this ensures the existence of a subsequence $\{\mu_E^{(T_k)}\}_{k\geq 1}$ which converges weakly to a limiting probability measure $\mu_E$ (since $\mu_E^{(T_k)}(\mathbb R)=1$). The estimate \eqref{equation squeezing measures} then shows that $\mu_E^{(T)}$ converges weakly to $\mu_E$ as $T\rightarrow \infty$, and thus
$$ \limsup_{Y\rightarrow \infty}  \frac 1Y \int_{2}^Y f(E(e^y)) dy=\liminf_{Y\rightarrow \infty}  \frac 1Y \int_{2}^Y f(E(e^y)) dy = \int_{\mathbb R} f(x) d\mu_E(x). $$
\end{proof}
\begin{remark}
Alternatively, we could have concluded the existence of a limiting distribution by applying Lemma 1.11 of \cite{Ell}, which asserts that if $\hat{\mu}_E^{(T)}(\xi)$ converges to a function uniformly in all compact subsets of $\mathbb R$, then $\mu_E^{(T)}$ converges weakly to a probability measure. We will see in Lemma \ref{lemma characteristic function converse} that (see also (21) of \cite{Sar})
 $$ \hat{\mu}^{(T)}_E(\xi) = e^{i\xi (2r_{\text{an}}(E)-1) \epsilon(\beta_0)} \prod_{ \substack{ \rho  \\ \Re(\rho)=\beta_0-\frac 12\\0<\Im(\rho) \leq T } } J_0\left( \frac{2 \xi}{|\rho|} \right), $$
which converges absolutely and uniformly in any compact subset of $\mathbb R$ to the function on the right-hand side of \eqref{equation characteristic function converse}.
\end{remark}

In the next lemma we give an explicit description of the Fourier Transform of $\mu_E$, which corresponds to (21) of \cite{Sar}.

\begin{lemma}
\label{lemma characteristic function converse}
Assume that $L(E,s)$ has at least one zero on the line $\Re(s)=\beta_0:=\sup\{ \Re(z) : L(E,z)=0 \}$ and assume that the set $\{ \Im(z)\geq 0 : L(E,z)=0, \Re(z)=\beta_0, z \neq 1 \}$ is linearly independent over $\Q$. Then the Fourier Transform of $\mu_E$ is given by
\begin{equation}
 \hat{\mu}_E(\xi) = e^{i\xi (2r_{\text{an}}(E)-1) \epsilon(\beta_0)} \prod_{ \substack{ \rho\\ \Re(\rho)=\beta_0-\frac 12\\\Im(\rho) >0 } } J_0\left( \frac{2 \xi}{|\rho|} \right) ,
 \label{equation characteristic function converse}
\end{equation}
where $\rho$ runs over the non-trivial zeros of $L(E,s+\frac 12)$ and
$$\epsilon(\beta_0) := \begin{cases}
  1 &\text{ if } \beta_0=1 \\
0 &\text{ otherwise.}
\end{cases}$$

\end{lemma}
\begin{proof}

We first compute the Fourier Transform of $\mu^{(T)}_E$. Note that the assumption that the set $\{ \Im(z) : L(E,z)=0, \Re(z)=\beta_0, z \neq 1 \}$ is linearly independent over $\mathbb Q$ implies that if $\beta_0>1$, then $L(E,\beta_0)\neq 0$. Hence, \eqref{equation F_T approximation} becomes
$$ F_T(x)= (2r_{\text{an}}(E)-1)\epsilon(\beta_0) + \sum_{\substack{\rho \\ \Re(\rho)=\beta_0-\frac 12 \\ 0<|\Im(\rho)|\leq T}} \frac{x^{i\gamma}}{\rho } +o_{x\rightarrow \infty}(1)+O_E(x^{-\delta_T} (\log T)^2),  $$
where $\rho$ runs over the non-trivial zeros of $L(E,s+\frac 12)$. Now, since $\mu^{(T)}_E$ is the limiting distribution of $F_T(e^y)$ by Lemma \ref{lemma converse truncated sum}, we deduce by classical arguments (see for instance the proof of Proposition 2.13 of \cite{FiMa}) that

$$ \hat{\mu}^{(T)}_E(\xi) = e^{i\xi (2r_{\text{an}}(E)-1) \epsilon(\beta_0)} \prod_{ \substack{ \rho  \\ \Re(\rho)=\beta_0-\frac 12\\0<\Im(\rho) \leq T } } J_0\left( \frac{2 \xi}{|\rho|} \right) .$$

The proof follows from the fact that the measures $\mu_E^{(T)}$ converge weakly to $\mu_E$ (this was established in the proof of Lemma \ref{lemma limiting distribution of E converse}), and thus L\'evy's criterion implies that $\hat{\mu}^{(T)}_E(\xi) \rightarrow \hat{\mu}_E(\xi)$ pointwise (see Lemma 1.11 of \cite{Ell}).

\end{proof}

We are now ready to prove Theorem \ref{theorem converse}.

\begin{proof}[Proof of Theorem \ref{theorem converse}]
Let $\{E_n\}_{n\geq 1}$ be a sequence of elliptic curves over $\mathbb Q$ for which LI($E_n$) holds and for which as $n\rightarrow \infty$,
$$ \overline{\delta}(E_n) \rightarrow 1. $$

Assume also that the Riemann Hypothesis does not hold for $L(E_n,s)$, that is $\beta_n:=\sup\{ \Re(z) : L(E_n,z)=0 \}>1$, for arbitrarily large values of $n$. By Lemmas \ref{lemma limiting distribution of E converse} and \ref{lemma characteristic function converse},
the quantity 
$$ E_n(x) := -\frac{\log x}{x^{\beta_n-\frac 12}} \sum_{p\leq x} \frac{a_p(E_n)}{\sqrt p}
$$
has a limiting logarithmic distribution $\mu_{E_n}$ whose Fourier transform is given, for the values of $n$ for which $\beta_n>1$, by 
\begin{equation}
 \hat{\mu}_{E_n}(\xi) = \prod_{ \substack{ \rho \\ \Re(\rho)=\beta_n-\frac 12\\\Im(\rho) >0 } } J_0\left( \frac{2 \xi}{|\rho|} \right), 
 \label{equation explicit fourier transform}
\end{equation}
where $\rho$ runs over the non-trivial zeros of $L(E_n,s+\frac 12)$. This implies that the limiting distribution of $E_n(e^y)$ is continuous, that is $\mu_{E_n}$ is absolutely continuous with respect to Lebesgue measure. Indeed, if $L(E_n,s+\frac 12)$ has a finite number of zeros $\rho$ on the line $\Re(\rho)=\beta_n-\frac 12$, then the bound $|J_0(x)|\leq \min(1,(\pi|x|/2)^{-\frac 12})$ (see (4.5) of \cite{RubSar}) implies that for $\xi \gg_n 1$,
$$ |\hat{\mu}_{E_n}(\xi)| \leq  \prod_{ \substack{ \rho \\ \Re(\rho)=\beta_n-\frac 12\\\Im(\rho) >0 } } (\pi |\xi|/|\rho|)^{-\frac 12} \ll_n |\xi|^{-\frac 12}, $$
and thus combining this with the bound $|\hat{\mu}_{E_n}(\xi)|\leq 1$, the absolute continuity of $\mu_{E_n}$ follows by applying Lemma 1.23 of \cite{Ell}. As for the case where $L(E_n,s+\frac 12)$ has infinitely many zeros $\rho$ on the line $\Re(\rho)=\beta_n-\frac 12$, we have by adapting the proof of Lemma 2.16 of \cite{FiMa} that for $\xi \gg_n 1$,
$$  |\hat{\mu}_{E_n}(\xi)| \leq  \prod_{ \substack{ \rho \\ \Re(\rho)=\beta_n-\frac 12\\0<\Im(\rho) <|\xi|/2-1 } } (\pi |\xi|/|\rho|)^{-\frac 12} \leq 2^{ -\# \{\rho: \Re(\rho)=\beta_n-\frac 12, 0<\Im (\rho) <|\xi|/2-1 \}/2}=o_{|\xi|\rightarrow \infty }(1), $$
and once more the absolute continuity of $\mu_{E_n}$ follows by applying Lemma 1.23 of \cite{Ell}.

Now, $J_0(t)$ is an even function, and thus so is $\hat{\mu}_{E_n}(\xi)$ by \eqref{equation explicit fourier transform}. Since $\hat{\mu}_{E_n}(\xi)$ is also real for real $\xi$, this implies that $\mu_{E_n}$ is symmetrical, and hence for arbitrarily large values of $n$ we have by absolute continuity of $\mu_{E_n}$ that
$$ \overline{\delta}(E_n)  = \underline{\delta}(E_n)=\delta(E_n) = \mu_{E_n}([0,\infty))=\frac 12, $$
contradicting our assumption that $\delta(E_n) \rightarrow 1$ as $n$ tends to infinity. Having reached a contradiction, we deduce that for all large enough values of $n$, the Riemann Hypothesis holds for $L(E_n,s)$.

We now show the existence of elliptic curves of arbitrarily large rank.  Assume that Conjecture \ref{conjecture unbounded ranks} is false, that is for all elliptic curves $E$ over $\mathbb Q$,
$$ r_{\text{an}}(E) \ll \sqrt{\log N_E}. $$
We know that the Riemann Hypothesis holds for $L(E_n,s)$ for large enough $n$, and hence all of the lemmas of Section \ref{section necessary condition} hold. In particular, taking $E$ to be any of these curves and denoting by $X_E$ the random variable associated with $\mu_E$, we have by Lemma \ref{lemma first two moments} that
\begin{equation}
\E[X_E]= 2r_{\text{an}}(E)-1, \hspace{1cm} \V[X_E] = \sum_{\gamma_E\neq 0} \frac{1}{\frac 14+\gamma_E^2} \asymp \log N_E,
\label{equation first two moments converse}
\end{equation} 
(see (27)-(29) of \cite{Sar}) since the assumption LI($E$) implies that the non-real zeros of $L(E,s)$ are simple. Moreover, the Riemann hypothesis for $L(E,s)$ implies that we have $\rho=\frac 12+i\gamma$ in Lemma \ref{lemma characteristic function converse}, that is
$$ \hat{\mu}_E(\xi) = e^{i\xi(2r_{\text{an}}(E)-1)}\prod_{ \substack{ \rho \\ \Im(\rho) >0 } } J_0\left( \frac{2 \xi}{\sqrt{\frac 14+\Im(\rho)^2}} \right). $$
Defining 
$$Y_E:=\frac{X_E-\E[X_E]}{\sqrt{\V[X_E]}},$$
where $X_E$ is the random variable associated with the measure $\mu_E$, we have by the analyticity of $\log J_0(z)$ in the disk\footnote{The function $\log J_0(z)$ is holomorphic in the disc $|z|<x_0$, where $x_0=2.4048\dots$ is the first zero of $J_0(z)$.} $|z|\leq \frac {12}5$ that taking Taylor series in the range $|t|\leq \sqrt{\V[X_E]}$,
\begin{align*}
\log \hat{Y}_E(t)&= \sum_{ \substack{ \rho \\ \Im(\rho) >0 } }  \log J_0\left( \frac{2 t}{\sqrt{\V[X_E]}\sqrt{\frac 14+\Im(\rho)^2}} \right) \\
& =  - \frac{(2t)^2}{4\V[X_E]} \sum_{ \substack{ \rho\\ \Im(\rho) >0 } } \frac 1{\frac 14+\Im(\rho)^2} +O\left(\frac{t^4}{(\log N_E)^2}  \sum_{ \substack{ \rho \\ \Im(\rho) >0 } } \frac 1{\left(\frac 14+\Im(\rho)^2\right)^2} \right)\\
&= -\frac{t^2}2 +O\left( \frac{t^4}{\log N_E}\right),
\end{align*} 
by the Riemann-von Mangoldt Formula. Hence, $\hat{Y}_E(t) \rightarrow e^{-t^2/2}$ pointwise as $N_E\rightarrow \infty$, and thus L\'evy's criterion (see Lemma 1.11 of \cite{Ell}) implies that $Y_E$ converges weakly to a Gaussian distribution. By the absolute continuity of $\mu_E$, we have that $\delta(E)$ exists, and
$$ \delta(E) = \P[X_E>0] = \P[Y_E> -\E[X_E]/\sqrt{\V[X_E]}  ]. $$
Therefore, the assumption that $r_{\text{an}}(E)\ll \sqrt{\log N_E}$ and \eqref{equation first two moments converse} imply that this last quantity is 
$$ \leq \P[Y_E> -C ],  $$
for some absolute constant $C$. By the central limit theorem we just proved, this quantity tends to 
$$ \frac 1{\sqrt {2\pi}}\int_{-C}^{\infty} e^{-\frac{t^2}2} dt<1 $$
as $N_E \rightarrow \infty$. Therefore, we obtain the bound 
$$\limsup_{N_E \rightarrow \infty} \delta(E) \leq \frac 1{\sqrt {2\pi}}\int_{-C}^{\infty} e^{-\frac{t^2}2} dt<1,$$ 
which contradicts our assumption that $\delta(E_n)=\overline{\delta}(E_n)\rightarrow 1$ as $n\rightarrow \infty$. Having reached a contradiction, we conclude that Conjecture \ref{conjecture unbounded ranks} holds.
\end{proof}

\begin{proof}[Proof of Theorem \ref{theorem converse RH}]
Following the steps of the proof of Theorem \ref{theorem converse}, one sees that if LI$(E)$ holds and $\beta_E>1$, then by the symmetry and absolute continuity of $\mu_E$ we have
$$ \overline{\delta}(E_n)  = \underline{\delta}(E_n)=\delta(E_n) = \mu_{E_n}([0,\infty))=\frac 12. $$
Hence, we have proved the contrapositive Theorem \ref{theorem converse RH}.
\end{proof}

\begin{remark}
One can weaken the second hypothesis of Theorem \ref{theorem converse} to
$$ \limsup_{N_E \rightarrow \infty} \sqrt{\log N_E}(\overline{\delta}(E) -\tfrac 12)=\infty, $$
and still conclude the unboundedness of the analytic rank of elliptic curves over $\mathbb Q$.

\end{remark}

\appendix
\section{Comparison of conjectures on large ranks of elliptic curves}

We conclude the paper with a numerical study of Conjecture \ref{conjecture unbounded ranks}, comparing the conjectures of Ulmer and Farmer, Gonek and Hughes. In this section we assume the Birch and Swinnerton-Dyer Conjecture, that is $r_{\text{an}}(E)=r_{\text{al}}(E)$ for every elliptic curve $E$ over $\mathbb Q$.

Mestre \cite{Me} has shown that 
$$ r_{\text{an}}(E)\ll \log N_E, $$
and under ECRH,
\begin{equation}
 r_{\text{an}}(E)\ll \frac{\log N_E}{\log\log N_E}.
 \label{equation Mestre bound}
\end{equation}
For elliptic curves over function fields, Ulmer has shown that the analogue of \eqref{equation Mestre bound} is best possible, and thus he conjectured that \eqref{equation Mestre bound} is also best possible for elliptic curves over $\mathbb Q$ \cite{Ul,Ul2}. Elkies and Watkins \cite{ElWa} have given numerical evidence for Ulmer's conjecture by finding elliptic curves having large rank and moderate conductor. They mention that numerical data shows that the statement
\begin{equation}
0< \limsup_{N_E\rightarrow \infty} \frac{r_{\text{an}}(E)}{\log N_E/\log\log N_E} < \infty 
\label{equation conjecture Ulmer}
\end{equation} 
is quite likely to be true. 
A few years later, Farmer, Gonek and Hughes \cite{FGH} constructed a random matrix model which suggests the conjecture

\begin{equation}
0< \limsup_{N_E\rightarrow \infty} \frac{r_{\text{an}}(E)}{\sqrt{\log N_E \log\log N_E}} < \infty.
\label{equation conjecture Farmer Gonek Hughes}
\end{equation}

Interestingly, Elkies and Watson's numerical data supports both \eqref{equation conjecture Ulmer} and \eqref{equation conjecture Farmer Gonek Hughes}. The reason for this is that the quotient between the two conjectures, that is 
$$ f(N_E):= \frac{ (\log N_E)^{\frac 12} }{ (\log\log N_E)^{\frac 32}}, $$
is contained in the interval $[0.86, 1.5]$ for all conductors $25\leq N_E\leq 10^{250}$, and hence it is impossible to decide which of \eqref{equation conjecture Ulmer} or \eqref{equation conjecture Farmer Gonek Hughes} is more likely to be true with the current data. 
Let us compare these conjectures with the elliptic curves appearing in \cite{ElWa}:

\begin{center}
\begin{tabular}{|c|c|c|c|c|}
\hline
$r_{\text{\text{al}}}(E)$ & $N_E$ & $\log N_E / \log\log N_E$ &  $\sqrt{\log N_E \log\log N_E}$ \\
\hline
5	& $1.9\cdot 10^7$	&6.874	&5.946\\
6	&$5.2\cdot 10^9$	&8.338	&7.198	\\
7	& $3.8\cdot 10^{11}$	&9.358	&8.122\\
8	& $2.5\cdot 10^{14}$	&10.773	&9.469	\\
9	& $3.2\cdot 10^{16}$	&11.759	&10.448	\\
10	& $1.0\cdot 10^{18}$	&12.861	&11.582	\\
11	& $1.8\cdot 10^{22}$	&14.203	&13.018	\\
\hline
\end{tabular}
\end{center}
and those appearing in \cite{Me}\footnote{The first column of this table is actually a lower bound on the rank, which can be shown to equal the rank under standard hypotheses.}:
\begin{center}
\begin{tabular}{|c|c|c|c|c|}
\hline
$r_{\text{\text{al}}}(E)$ & $N_E$ & $\log N_E / \log\log N_E$ &  $\sqrt{\log N_E \log\log N_E}$ \\
\hline
3&	$5.1\cdot 10^3$	&4.277	&3.980\\
4&	$5.4 \cdot 10^5$	&5.839	&5.118	\\
5&	$1.7\cdot 10^8$	&7.482	&6.455	\\
6&	$5.1\cdot 10^{10}$	&8.893	&7.696	\\
7&	$3.2\cdot 10^{12}$	&9.841	&8.574	\\
8&	$1.8\cdot 10^{15}$	&11.181	&9.870	\\
9&	$7.0\cdot 10^{19}$	&13.215	&11.956	\\
10&	$5.2\cdot 10^{22}$	&14.386	&13.217	\\
11	&$1.8\cdot 10^{24}$	&14.989	&13.884	\\
12	&$2.7\cdot 10^{29}$	&16.903	&16.073	\\
13	&$2.1\cdot 10^{38}$	&19.885	&19.699 \\
14	&$3.6\cdot 10^{37}$	&19.640	&19.391	\\
\hline
\end{tabular}
\end{center}
Both of these tables show that \eqref{equation conjecture Ulmer} and \eqref{equation conjecture Farmer Gonek Hughes} are equally likely to be true.

\section*{Acknowledgements}

I would like to thank Enrico Bombieri, Andrew Granville, Jeffrey Lagarias and Peter Sarnak for their support and valuable advice. I also thank Barry Mazur for motivating me to weaken the Linear Independence Hypothesis in the necessary condition. I thank Jeffrey Lagarias for suggesting me to provide an equivalence rather than a one-sided implication. Finally, I thank Byungchul Cha, Jan-Christoph Schlage-Puchta and Anders S\"odergren for their comments and for fruitful conversations, and the referee for his comments. This work was supported by an NSERC Postdoctoral Fellowship, as well as NSF grant DMS-0635607, and was accomplished  at the Institute for Advanced Study and at the University of Michigan.

\end{document}